\documentclass[a4paper,12pt]{amsart}
\usepackage{fouriernc}
\usepackage[T1]{fontenc}

\usepackage{graphicx}
\usepackage{amsmath}
\usepackage{amssymb}
\usepackage[active]{srcltx}
\usepackage{hyperref}
\usepackage{bbm}
\usepackage{enumerate}

\usepackage{geometry}
\newtheorem*{thm*}{Theorem}
\newtheorem{prop}{Proposition}%
\newtheorem*{cor*}{Corollary}%[section]
\theoremstyle{definition}

\theoremstyle{remark}
\theoremstyle{plain}

\def\CC{{\mathbb C}}

\def\NN{{\mathbb N}}
\def\QQ{{\mathbb Q}}

\def\RR{{\mathbb R}}
\def\TT{{\mathbb T}}

\def\ZZ{{\mathbb Z}}

\def\hatZZ{\widehat{\mathbb Z}}

\def\vecj{{\text{\boldmath$j$}}}

\def\vecm{{\text{\boldmath$m$}}}

\def\vecp{{\text{\boldmath$p$}}}
\def\vecr{{\text{\boldmath$r$}}}

\def\vecu{{\text{\boldmath$u$}}}
\def\vecv{{\text{\boldmath$v$}}}

\def\vecx{{\text{\boldmath$x$}}}

\def\vecy{{\text{\boldmath$y$}}}

\def\vecalf{{\text{\boldmath$\alpha$}}}

\def\vecomega{{\text{\boldmath$\omega$}}}

\def\vecnull{{\text{\boldmath$0$}}}

\def\scrA{{\mathcal A}}
\def\scrB{{\mathcal B}}

\def\scrD{{\mathcal D}}

\def\scrF{{\mathcal F}}

\def\scrS{{\mathcal S}}
\def\scrT{{\mathcal T}}

\def\fB{{\mathfrak B}}
\def\fC{{\mathfrak C}}

\def\Re{\operatorname{Re}}

\def\dist{\operatorname{dist}}

\def\SL{\operatorname{SL}}

\def\SO{\operatorname{SO}}

\def\vol{\operatorname{vol}}

\def\GamG{\Gamma\backslash G}

\def\trans{\,^\mathrm{t}\!}

\numberwithin{equation}{section}

\title{Smallest denominators}
\author{Jens Marklof}
\address{Jens Marklof, School of Mathematics, University of Bristol, Bristol BS8 1UG, U.K.\newline \rule[0ex]{0ex}{0ex} \hspace{8pt}{\tt j.marklof@bristol.ac.uk}}
\date{17 October 2023/1 February 2024}
\thanks{Research supported by EPSRC grant EP/S024948/1. Data supporting this study are included within the article. MSC (2020): 11K60, 11J13, 37A17}

\begin{document}

\begin{abstract} 
We establish higher dimensional versions of a recent theorem by Chen and Haynes [Int. J. Number Theory 19 (2023), 1405--1413] on the expected value of the smallest denominator of rational points in a randomly shifted interval of small length, and of the closely related 1977 Kruyswijk-Meijer conjecture recently proved by Balazard and Martin [Bull. Sci. Math. 187 (2023), Paper No. 103305]. We express the distribution of smallest denominators in terms of the void statistics of multidimensional Farey fractions and prove convergence of the distribution function and certain finite moments. The latter was previously unknown even in the one-dimensional setting. We furthermore obtain a higher dimensional extension of Kargaev and Zhigljavsky's work on moments of the distance function for the Farey sequence [J. Number Theory 65 (1997), 130--149] as well as new results on pigeonhole statistics.
\end{abstract}

\maketitle

\section{Introduction (the one-dimensional case)}

Motivated by Meiss and Sander's recent paper \cite{Meiss21} (which we will return to in Section \ref{minsec}), Chen and Haynes \cite{Chen23} investigated the smallest denominator of all fractions in a small interval of length $\delta$ with random center $x$,
\begin{equation}\label{oneptone}
q_{\min}(x,\delta) = \min \left\{ q\in\NN : \exists \tfrac{p}{q} \in\QQ \cap (x-\tfrac{\delta}{2},x+\tfrac{\delta}{2}) \right\}.
\end{equation}
Their main results are (a) an explicit formula for the distribution for fixed $\delta$ and $x$ uniformly distributed in the unit interval, and (b) the asymptotics of the expectation of $q_{\min}(x,\delta)$ as $\delta\to 0$, which they show is $\frac{16}{\pi^2} \delta^{-1/2}+O(\log^2\delta)$. We will see below that the statistics of $q_{\min}(x,\delta)$ is in fact given by a scaled version of the Hall distribution for the gaps between Farey fractions. This complements recent work of Artiles \cite{Artiles23} who proved the existence of the limit distribution using dynamics on the space of lattices. (We will comment on the link between the two approaches at the end of this introduction.)

Farey fractions of level $Q$ are defined as the finite set
\begin{equation}\label{fareydef234}
	\scrF_Q=\bigg\{ \tfrac{p}{q} \in[0,1) : (p,q)\in\hatZZ^2, \; 0<q\leq Q \bigg\} ,
\end{equation}
where $\hatZZ^2$ denotes the set of primitive lattice points, i.e., integer vectors with coprime coordinates.
The number of elements is asymptotically $\#\scrF_Q \sim \sigma_Q := \frac{3}{\pi^2} Q^2$ as $Q\to\infty$. 
The Hall distribution $H(s)$ \cite{Hall70} describes the relative frequency of gaps in $\scrF_Q$ of size larger than $s\sigma_Q^{-1}$ as $Q\to\infty$; see \cite{Farey} for the relevant background. We have the explicit formula
\begin{equation}\label{formula}
H(s) = 
\begin{cases}
1 & \text{if $t\in[1,\infty)$}\\
-1+2t-2t\log t & \text{if $t\in[\frac14,1]$}\\
-1+ 2t+2\sqrt{\frac14-t}-4t\log\left(\frac12+\sqrt{\frac14-t}\right) & \text{if $t\in[0,\frac14]$,}
\end{cases}
\end{equation}
with shorthand $t=(\tfrac{\pi^2}{3} s)^{-1}$. Formula \eqref{formula} was rediscovered by Kargaev and Zhigljavsky in their study of the void distribution of $\scrF_Q$ \cite[Theorem 1.2 and Lemma 2.6]{Kargaev97}. There is now an extensive literature on the statistical properties of Farey fractions in dimension one, see \cite{Boca05} and references therein.

\begin{figure}
\begin{center}
\includegraphics[width=0.7\textwidth]{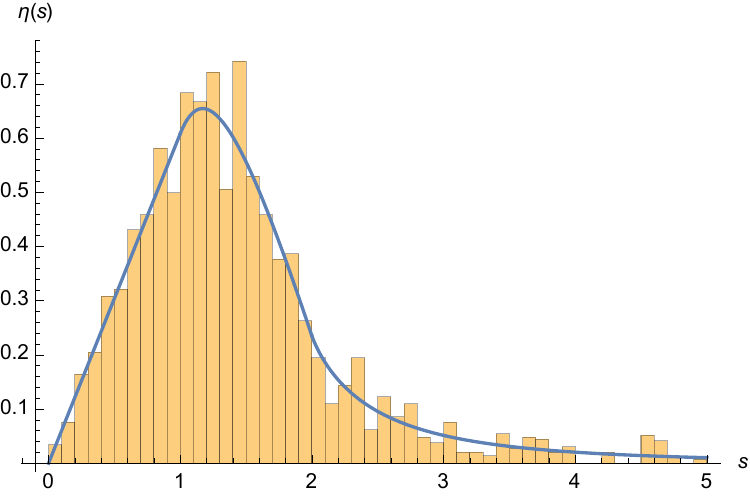}
\end{center}
\caption{The limit density $\eta(s)$ compared to the distribution of the smallest denominator of rationals in each interval $[\frac{j}{3000},\frac{j+1}{3000})$, $j=0,\ldots,2999$, cf.~Section \ref{secDiscrete}. The same law describes the shortest cycle length of a large random circulant directed graph of (in- and out-) degree 2  \cite{circulant}.} \label{fig1}
\end{figure}

Our first observation is the following. 

\begin{prop}\label{1111}
For any interval $\scrD\subset[0,1]$ and $L>0$, we have
\begin{equation}\label{LD1}
\lim_{\delta\to 0} \vol\left\{ x\in \scrD :  \delta^{1/2} q_{\min}(x,\delta) > L  \right\} = \vol\scrD \; \int_L^\infty \eta(s)\, ds
\end{equation}
with the probability density
\begin{equation}\label{Hallex}
\eta(s)=\tfrac{6}{\pi^2}\, s\, H(\tfrac{3}{\pi^2} s^2) .
\end{equation}
\end{prop}

\begin{proof}
We have
\begin{equation} \label{cal1}
q_{\min}(x,\delta) > L \delta^{-1/2} \Leftrightarrow \left\{ (p,q)\in\hatZZ^2 :  0< q \leq L\delta^{-1/2} , \tfrac{p}{q} \in \left(x-\tfrac{\delta}{2},x+\tfrac{\delta}{2}\right) \right\} =\emptyset  . 
\end{equation}
Now, for the choice $Q=L\delta^{-1/2}$, $s= \frac{3}{\pi^2} L^2$,
the right hand side of \eqref{cal1} is equivalent to
\begin{equation}\label{cal2}
\scrF_Q \cap \left(x-\tfrac{s}{2\sigma_Q},x+\tfrac{s}{2\sigma_Q}\right)+\ZZ =\emptyset .
\end{equation}
%and thus
%\begin{equation}\label{cal2}
%\left\{ x\in \scrD :  q_{\min}(x,\delta) > L \delta^{-1/2} \right\} = \left\{ x\in\scrD :  \scrF_Q \cap (x-\tfrac{s}{2\sigma_Q},x+\tfrac{s}{2\sigma_Q})+\ZZ =\emptyset \right\}.
%\end{equation}
As proved in \cite{Kargaev97} for $\scrD=[0,1]$, and in \cite{Farey} for general $\scrD$, the Lebesgue measure of the set of $x\in\scrD$ satisfying \eqref{cal2} has a limit as $Q\to\infty$, namely the void distribution $P(0,[-\frac{s}{2},\frac{s}{2}])=P(0,[0,s])$ in the notation of \cite{Farey}. Note that the limit is independent of the choice of $\scrD$. It is a general fact that the derivative of the void distribution yields the gap distribution \cite{nato}, 
\begin{equation}\label{diff}
-\frac{d}{ds} P(0,[0,s]) = P_0(0,[0,s]),
\end{equation}
which in the present case is the classic Hall distribution $H(s)$ \cite{Farey}. Both functions are continuous and we have $P(0,[0,\infty))=0$, so by the fundamental theorem of calculus
\begin{equation}
P(0,[0,s]) = \int_s^\infty P_0(0,[0,s'])\, ds'.
\end{equation}
The limit in \eqref{LD1} is therefore
\begin{equation}
\int_{\frac{3}{\pi^2} L^2}^\infty H(s)\, ds,
\end{equation}
and the formula for the limit density $\eta(s)$ follows by differentiation.
\end{proof}

From \eqref{formula} and \eqref{Hallex} we deduce the explicit formula
\begin{equation}\label{formula2}
\eta(s) = \tfrac{6}{\pi^2}\times
\begin{cases}
s  & \text{if $s\in[0,1]$}\\
-s+2s^{-1}+4s^{-1}\log s & \text{if $s\in[1,2]$}\\
-s+ 2s^{-1}+2 s \sqrt{\frac14 - s^{-2}}-4s^{-1}\log\bigg(\frac12+\sqrt{\frac14-s^{-2}}\bigg) & \text{if $s\geq 2$,}
\end{cases}
\end{equation}
see Figure \ref{fig1}. Note that $H(s) \sim\frac{18}{\pi^4} s^{-2}$ for $s$ large, and hence $\eta(s)=\tfrac{6}{\pi^2}\, s\, H(\tfrac{3}{\pi^2} s^2) \sim \frac{12}{\pi^2} s^{-3}$. 

Interestingly, $\eta(s)$ also describes the distribution of the shortest cycle length of a large random circulant directed graph of (in- and out-) degree 2  \cite[Figure 5 and Eq.~(5.19)]{circulant}. 

We extend Proposition \ref{1111} to rational points in arbitrary dimensions in Section \ref{secArb}, Proposition \ref{PropLD}. The key ingredients here are limit theorems for the fine-scale statistics of multidimensional Farey fractions \cite{Farey}.

\begin{figure}
\begin{center}
\includegraphics[width=0.7\textwidth]{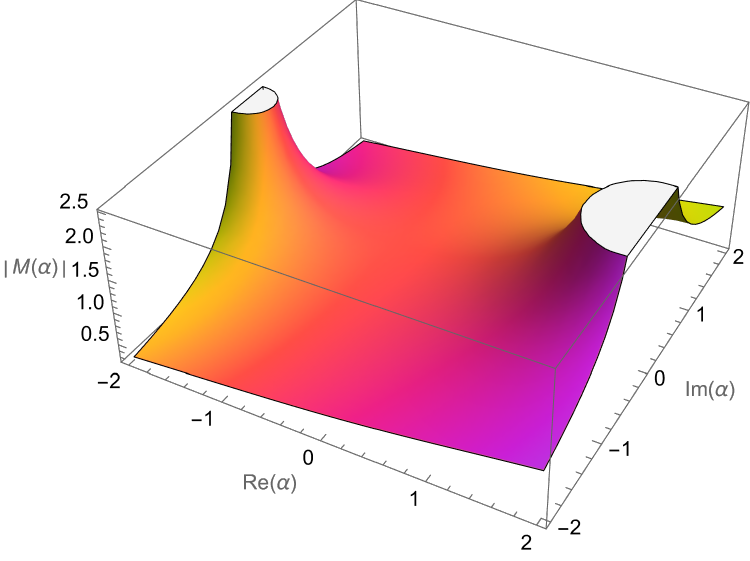}\quad\includegraphics[width=0.085\textwidth]{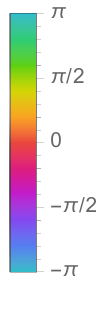}
\end{center}
\caption{The function $M(\alpha)$, with the height of the graph representing its absolute value and the colour its argument.} \label{fig2}
\end{figure}

Our next result is an extension of the Chen-Haynes asymptotics for the expected value to general (but small) moments. 

\begin{prop}\label{propo}
For any interval $\scrD\subset[0,1]$ and $\alpha\in\CC$ with $|\Re\alpha|<2$, we have
\begin{equation}\label{LD1Mom}
\lim_{\delta\to 0} \delta^{\alpha/2}  \int_{\scrD}   q_{\min}(x,\delta)^\alpha dx=
\vol\scrD \;  M(\alpha), \quad \text{with}\quad M(\alpha) = \int_0^\infty s^{\alpha} \eta(s)\, ds.
\end{equation}
\end{prop}

It is interesting that the convergence of moments smallest denominators is not an immediate corollary of the convergence of moments for the void distribution of Farey fractions proved in \cite{Kargaev97}, even though the limits coincide. We will prove Proposition \ref{propo} as a special case of Proposition \ref{propi} (valid in any dimension) in Section \ref{secArb}. Ref.~\cite{Kargaev97} provides explicit formulas and asymptotics for the moments of the void statistics. In particular, for $|\Re\alpha|<2$, these yield
\begin{equation}\label{lastin}
M(\alpha) = \frac{3}{\pi^2} \int_0^\infty t^{-(\alpha+4)/2} F(t) dt = \frac{6}{\pi^2(\alpha+2)} \int_0^1 t^{-(\alpha+2)/2} dF(t) ,
\end{equation}
where $F(t) = H(s)$ is the function on the right hand side of \eqref{formula}. The last integral in \eqref{lastin} is computed in \cite[Lemma 2.6]{Kargaev97}, and we obtain for $|\Re\alpha|<2$, $\alpha\neq 0$,
\begin{equation}
M(\alpha)  = \frac{24}{\pi^2 \alpha (\alpha+2)} \left( \frac{2}{\alpha} + 2^{\alpha} \mathrm{B}\left(-\frac{\alpha}{2},\frac12\right)\right)  ,
\end{equation}
where $\mathrm{B}(x,y)$ is the beta function (Euler's integral of the first kind); see Figure \ref{fig2} for a plot of $M(\alpha)$, and Figure \ref{fig4} in Section \ref{secDiscrete} for a comparison with numerical data.
For $\alpha=1$ the above expression evaluates to $\frac{16}{\pi^2}$, which is the constant found by Chen and Haynes \cite{Chen23}.

The remainder of this paper is organised as follows.
Following the same argument as in dimension one outlined above, we translate in Section \ref{secArb} the problem of smallest denominators in small subsets of $\RR^n$ to the statistics of multidimensional Farey fractions. We then apply the setting in \cite{Farey} and use equidistribution and escape-of-mass estimates for group actions on the space of lattices. 
For the distribution for rationals in sets with random center, the relevant action on the space of lattices is the $\RR^n$-action by the horospherical subgroup. If we move to the setting of the Kruyswijk-Meijer conjecture \cite{KM77,Steward13}, which was proved recently by Balazard and Martin \cite{Balazard23}, then the Lebesgue integral is replaced by a discrete average and, as we will explain in Section \ref{secDiscrete}, the relevant action is a $\ZZ^n$-action by the ``time-one'' map of the horospherical subgroup. This leads to the proof of convergence of the full distribution function, and we will see that the limit is the same as in the case of continuous sampling. It also provides an alternative proof of the Kruyswijk-Meijer conjecture, including extensions to other moments and again to higher dimensions. The role of the void statistics for Farey fractions is now replaced by the so-called pigeonhole statistics, which is of independent interest and the content of Section \ref{secPigeon}. We will use a similar strategy of proof as recently employed by Pattison \cite{Pattison23} for the pigeonhole statistics of $\sqrt{n}\bmod 1$. In Section \ref{sec97} we discuss moments of the distance function for the multidimensional Farey sequence, thus extending results of Kargaev and Zhigljavsky \cite{Kargaev97}.
Section \ref{minsec} concludes this study with a limit theorem for the Meiss-Sander distribution \cite{Meiss21} for minimal resonances in volume-preserving maps, which was the original motivation for Chen and Haynes \cite{Chen23}.

The interpretation of smallest denominators in terms of the space of lattices was recently pointed out by Artiles \cite{Artiles23}. Artiles proves convergence of the distribution function (an analogue of Proposition \ref{PropLD}), using the strategy developed by Str\"ombergsson and the author in \cite{partI} for more general lattice point problems concerning thin randomly sheared or rotated domains. Ref.~\cite{partI} includes an application to directional statistics for visible lattice points in arbitrary dimension and also formed the basis for the study of multidimensional Farey fractions in \cite{Farey}. The equivalence \eqref{cal1}-\eqref{cal2} (cf.~also \eqref{cal10} below) thus explains the link between \cite{Artiles23} and \cite{Farey}, both of which use equidistribution of closed horospheres to establish limit theorems for smallest denominators and Farey statistics, respectively. For a generalisation of the results in \cite{Farey} to Farey fractions subject to congruence conditions see Heersink \cite{Heersink21}. More restrictive constraints related to thin groups are discussed in the work of Lutsko \cite{Lutsko22}.

The present paper does not include any discussion of rates of convergence, although there is no principal obstruction in obtaining these since the horospherical equidistribution results we use here are available with precise error terms.

\section{Smallest denominators for multidimensional fractions}\label{secArb}

Define the set of $n$-dimensional Farey fractions of level $Q\geq 1$ ($Q$ not necessarily an integer) by
\begin{equation}\label{fareydef}
	\scrF_Q=\bigg\{ \frac{\vecp}{q} \in[0,1)^n : (\vecp,q)\in\hatZZ^{n+1}, \; 0<q\leq Q \bigg\} .
\end{equation}
For large $Q$, we have
\begin{equation}\label{asymQ}
	\#\scrF_Q \sim \sigma_Q:=\frac{Q^{n+1}}{(n+1)\,\zeta(n+1)}  .
\end{equation}
Given a bounded set $\scrA\subset\RR^n$ with boundary of Lebesgue measure zero and non-empty interior, define
\begin{equation}
q_{\min}(\vecx,\delta,\scrA) = \min \left\{ q\in\NN : \exists \tfrac{\vecp}{q} \in\QQ^n \cap \vecx+\delta\scrA \right\}.
\end{equation}
Assuming $\scrA$ has non-empty interior ensures the minimum exists.

Set furthermore $G=\SL(n+1,\RR)$, $\Gamma=\SL(n+1,\ZZ)$, and
\begin{equation}
P(0,\scrA) = \mu\{ g\in\GamG : \hatZZ^{n+1}g \cap \fC(\scrA)  = \emptyset\} ,
\end{equation}
where $\mu$ is the Haar probability measure on $\GamG$ and 
\begin{equation}
\fC(\scrA)=\{ (\vecx,y)\in\RR^n\times (0,1] : \vecx \in \sigma_1^{-1/n} y\scrA \} \subset \RR^{n+1}
\end{equation}
is a cone with base $\scrA$.

\begin{prop}\label{PropLD}
For $\scrA\subset\RR^n$ bounded and $\scrD\subset[0,1]^n$, both with boundary of Lebesgue measure zero and non-empty interior, $L>0$, we have
\begin{equation}\label{LD1zzz}
\lim_{\delta\to 0} \frac{\vol\left\{ \vecx\in \scrD :  \delta^{n/(n+1)} q_{\min}(\vecx,\delta,\scrA) > L  \right\}}{\vol\scrD } = E_\scrA(L)
\end{equation}
with $E_\scrA(L)=P(0,\sigma_1^{1/n} L^{1+1/n} \scrA)$.
\end{prop}

\begin{proof}
By the same token as in the one-dimensional case, we have
\begin{equation} \label{cal10}
q_{\min}(\vecx,\delta,\scrA) > L \delta^{-n/(n+1)} \Leftrightarrow \scrF_Q \cap \vecx+\sigma_Q^{-1/n} s \scrA+\ZZ^n =\emptyset ,
\end{equation}
with $Q=L \delta^{-n/(n+1)}$ and $s=\sigma_1^{1/n} L^{1+1/n}$.
%We have here again used the fact that, for any given $\vecx,\delta, L$, there exists $(\vecp,q)\in\hatZZ^{n+1}$ with $q\geq L\delta^{-1/2}$ such that $\tfrac{\vecp}{q} \in \vecx+ \delta \scrA$ (this holds since $\scrA$ has non-empty interior). 
Theorem 3 in \cite{Farey} (which is based on the results in \cite{partI}; see also Proposition \ref{PropFarey1} in Section \ref{secPigeon}) states that the volume of the set of $\vecx\in\scrD$ satisfying \eqref{cal10} converges to $P(0,s \scrA)$.
\end{proof}

In dimension $n>1$ we have no explicit expressions for $E_\scrA(L)$. We know it is continuous in $L$, and continuously differentiable if $\scrA$ is a ball, see \cite[Remark 2.6]{partI}. If $\scrA$ is a fixed ball, we have $P(0,s\scrA)\asymp s^{-n}$ as $s\to\infty$, see \cite[Section 1.3]{Strombergsson11}. We can obtain upper (resp. lower) tail estimates for general bounded $\scrA$ with non-empty interior by using a ball that is contained in (resp. contains)  $\scrA$. This yields
$E_\scrA(L)\asymp L^{-(n+1)}$ for $L\to\infty$.
For $n=1$ this is consistent with the tail of $E_{(-\frac12,\frac12)}(L)=\int_L^\infty \eta(s)\, ds$.

Let us now turn to the convergence of moments. 
%The below statement should hold for all positive moments $\alpha <n+1$ (for these $\alpha$ the limit distribution is finite, in view of the above tail bounds) but we restrict here to the shorter range $\alpha <1+1/n$, to keep the proof simple.

\begin{prop}\label{propi}
For $\scrA\subset\RR^n$ bounded and $\scrD\subset[0,1]^n$, both with boundary of Lebesgue measure zero and non-empty interior, $\alpha\in\CC$ with $|\Re\alpha| <n+1$, we have
\begin{equation}\label{LD1Mom2345}
\lim_{\delta\to 0} \frac{\delta^{\alpha n/(n+1)} }{\vol\scrD} \int_{\scrD}   q_{\min}(\vecx,\delta,\scrA)^\alpha d\vecx=
\int_0^\infty L^{\alpha} \; dE_\scrA(L).
\end{equation}
\end{prop}

Note that the For $n=1$, Proposition \ref{propi} specialises to Proposition \ref{propo}.
The proof will require the notion of a Siegel set defined as follows.
For $$\vecu=(u_{12},\ldots,u_{1(n+1)},u_{23},\ldots,u_{2(n+1)},\ldots,u_{n(n+1)})\in\RR^{n(n+1)/2}$$ and 
$$\vecv=(v_1,v_2,\ldots,v_{n+1}) \in \scrT, \quad \text{with}\quad \scrT=\{(v_1,\ldots,v_{n+1})\in\RR_{>0}^{n+1}, v_1\cdots v_{n+1}=1\},$$ 
let
\begin{equation}\label{uvdef}
n(\vecu):=\left(\begin{matrix} 1 & u_{12} & \cdots & u_{1(n+1)}\\ & \ddots & & \vdots \\ & & 1 & u_{n(n+1)} \\ & & & 1 \end{matrix}\right),\quad a(\vecv):=\left(\begin{matrix} v_1 &  &  & \\ & v_2 & &  \\ & & \ddots & \\ & & & v_{n+1} \end{matrix}\right).
\end{equation}
The Iwasawa decomposition of $g\in G$ is then given by
\begin{equation} \label{Iwa}
g=n(\vecu)a(\vecv)k,
\end{equation}
where $\vecu\in\RR^{n(n+1)/2}$, $\vecv\in\scrT$ and $k\in \SO(n+1)$. 
The Siegel set
\begin{equation}\label{Siegel}
\scrS_\Gamma:=\left\{ n(\vecu)a(\vecv)k: \vecu\in[-\tfrac{1}{2},\tfrac{1}{2}]^{n(n+1)/2}, \; 0<v_{j+1}\leq\frac{2}{\sqrt{3}}v_{j} ,\; k\in \SO(n+1) \right\} \subset G
\end{equation}
has the property that it contains a fundamental domain $\scrF_\Gamma\subset G$ of the $\Gamma$-action and can be covered with a finite number of $\Gamma$-translates of $\scrF_\Gamma$. We fix $\scrF_\Gamma$ and set $v_j(\Gamma g)=v_j(g)=v_j$, with $g=n(\vecu)a(\vecv)k\in\scrF_\Gamma$. 
%For $L>0$, we will furthermore define the subset
%\begin{equation}\label{SiegelL}
%\scrS_\Gamma(L):=\left\{ n(\vecu)a(\vecv)k \in\scrS_\Gamma : v_1 \geq L \right\} .
%\end{equation}

\begin{proof}[Proof when $\Re\alpha=0$.]
Proposition \ref{PropLD} implies (and, given the continuity of $E_\scrA(L)$ in $L$, is in fact equivalent to) the statement that for any bounded continuous function $h:\RR_{\geq 0} \to \CC$,
\begin{equation}\label{LD1Mom23450001}
\lim_{\delta\to 0} \frac{1}{\vol\scrD} \int_{\scrD}   h\left(\delta^{n/(n+1)}q_{\min}(\vecx,\delta,\scrA)\right) d\vecx=
\int_0^\infty h(L) \; dE_\scrA(L).
\end{equation}
Now take $h(x)=x^\alpha$ and the claim is proved.
\end{proof}

\begin{proof}[Proof when $\Re\alpha>0$.]
We have
\begin{equation}
\delta^{\alpha n/(n+1)} \int_{\scrD} q_{\min}(\vecx,\delta,\scrA)^\alpha d\vecx = 
 \alpha \int_0^\infty L^{\alpha-1} \vol\left\{ \vecx\in \scrD :  \delta^{n/(n+1)} q_{\min}(\vecx,\delta,\scrA) > L \right\} dL .
\end{equation}
In view of Proposition \ref{PropLD}, for any $R>r>0$,
\begin{equation}
\lim_{\delta\to 0} \int_r^R L^{\alpha-1} \vol\left\{ \vecx\in \scrD :  \delta^{n/(n+1)}  q_{\min}(\vecx,\delta,\scrA) > L \right\} dL
\\ =
\vol\scrD \;  \int_r^R L^{\alpha-1} E_\scrA(L)\, dL.
\end{equation}
Therefore, all that remains to be shown is that
\begin{equation}\label{the1st}
\lim_{R\to\infty} \limsup_{\delta\to 0} \int_R^\infty L^{\Re\alpha-1} \vol\left\{ \vecx\in \scrD :  \delta^{n/(n+1)}  q_{\min}(\vecx,\delta,\scrA) > L \right\} dL = 0,
\end{equation}
\begin{equation}\label{the2nd}
\lim_{r\to 0} \limsup_{\delta\to 0} \int_0^r L^{\Re\alpha-1} \vol\left\{ \vecx\in \scrD :  \delta^{n/(n+1)}  q_{\min}(\vecx,\delta,\scrA) > L \right\} dL = 0.
\end{equation}
Relation \eqref{the2nd} is immediate since the integrand is bounded above by $L^{\Re\alpha-1}$.
We will establish \eqref{the1st} by proving that there is a constant $C$ such that for all $\delta>0$, $L\geq 1$, we have
\begin{equation}\label{bibo}
 \vol\left\{ \vecx\in [0,1]^n :  \delta^{n/(n+1)}  q_{\min}(\vecx,\delta,\scrA) > L  \right\}
 \leq C L^{-(n+1)} .
\end{equation}
To this end, recall the observation \eqref{cal10} and furthermore (the starting point of \cite{Farey}) that
\begin{equation} \label{cal1000}
\scrF_Q \cap \vecx+\sigma_Q^{-1/n} s \scrA+\ZZ^n =\emptyset 
\Leftrightarrow  \hatZZ^{n+1}h(\vecx) a(Q) \cap \fC(s\scrA) =\emptyset ,
\end{equation}
where
\begin{equation}
h(\vecx)=\begin{pmatrix} 1_n & \trans\vecnull \\ -\vecx & 1 \end{pmatrix},\qquad a(y)=\begin{pmatrix} y^{1/n} 1_n & \trans\vecnull \\ \vecnull & y^{-1} \end{pmatrix} .
\end{equation}
With the choice $Q=L \delta^{-n/(n+1)}$ and $s=\sigma_1^{1/n} L^{1+1/n}$ this becomes
\begin{equation} \label{truu}
\hatZZ^{n+1}h(\vecx) a(\delta^{-n/(n+1)}) \cap \fC(\sigma_1^{1/n} L^{1+1/n} \scrA) a(L^{-1}) =\emptyset ,
\end{equation}
where we note that
\begin{equation}
\fC(\sigma_1^{1/n} L^{1+1/n} \scrA) a(L^{-1}) = L \fC(\sigma_1^{1/n} \scrA) 
\end{equation}
is the homothetic dilation by $L$ of the fixed cone $\fC(\sigma_1^{1/n} \scrA)$. 

Since $\fC(\sigma_1^{1/n} \scrA)$ is a cone with vertex at the origin, relation \eqref{truu} is equivalent to
\begin{equation} \label{truu22}
(\ZZ^{n+1}\setminus\{0\})\; h(\vecx) a(\delta^{-n/(n+1)}) \cap L \fC(\sigma_1^{1/n} \scrA)  =\emptyset .
\end{equation}
Because $\scrA$ has non-empty interior, $\fC(\sigma_1^{1/n} \scrA)$ contains an open ball $\scrB_0$ of radius $r_0>0$ not containing the origin, and hence also $L\scrB_0\subset L\fC(\sigma_1^{1/n} \scrA)$. Now, \cite[Lemma 2.1]{Strombergsson11} tells us, given $r_0$ there is a constant $r_1>0$ such that for all $L>0$, we have $v_1(g) \geq r_1 L$ for any lattice $\ZZ^{n+1} g$ with $g=n(\vecu)a(\vecv)k\in\scrS_\Gamma$ which does not intersect a ball of radius $r_0 L$. The left hand side of \eqref{bibo} is thus bounded above by
\begin{equation}\label{bibo2}
 \vol\left\{ \vecx\in [0,1]^n :  v_1\big(\Gamma h(\vecx) a(Q)\big) \geq r_1 L \right\} .
\end{equation}
An upper bound for \eqref{bibo2} follows from the proof of Proposition 5.1 (case B1) in \cite{Kim23}. For $1\leq s\leq n$ and $\underline{l}=(l_1,\cdots,l_s)\in\ZZ^s_{\ge0}$, we set (cf.~\cite[(5.5)]{Kim23})
\begin{equation} \label{Xilsdef}
\Xi^s_{\underline{l}}:=\left\{g\in\scrF_\Gamma :s(g)=s, \delta_{n+1} 2^{l_i}<  v_i(g)\leq\delta_{n+1} 2^{l_i+1} \;(i=1,\cdots,s)\right\}
\end{equation}
with $\delta_d=d4^d$ and $s(g)$ is the largest $i$ for which $v_i(g)> 1$. With this, the estimate leading to \cite[(5.21)]{Kim23} shows that \eqref{bibo2} is bounded above by
\begin{equation}\label{bibo3}
\begin{split}
\sum_{s=1}^n \sum_{\substack{\underline{l}\in\ZZ^s_{\ge0}\\ \delta_{n+1} 2^{l_1+1} \geq r_1 L }}
&  \vol\left\{ \vecx\in [0,1]^n :  \exists\gamma\in\Gamma \text{ s.t. }\gamma h(\vecx) a(Q) \in \Xi^s_{\underline{l}} \right\} \\
 &\ll  \sum_{s=1}^n \sum_{\substack{\underline{l}\in\ZZ^s_{\ge0}\\ \delta_{n+1} 2^{l_1+1} \geq r_1 L}} \prod_{i=1}^s 2^{-(n+1)l_i} \ll L^{-(n+1)},
 \end{split}
\end{equation}
and therefore
\begin{equation}\label{bibo22}
 \vol\left\{ \vecx\in [0,1]^n :  v_1\big(\Gamma h(\vecx) a(Q)\big) \geq r_1 L \right\} \ll L^{-(n+1)}.
\end{equation}
This yields \eqref{bibo} and the proof for positive $\Re\alpha$ is complete. 
\end{proof}

\begin{proof}[Proof when $\Re\alpha<0$.]
We now write
\begin{equation}
\delta^{\alpha n/(n+1)} \int_{\scrD} q_{\min}(\vecx,\delta,\scrA)^\alpha d\vecx = 
 -\alpha \int_0^\infty L^{\alpha-1} \vol\left\{ \vecx\in \scrD :  \delta^{n/(n+1)}  q_{\min}(\vecx,\delta,\scrA) \leq L \right\} dL .
\end{equation}
The argument is analogous to the previous case of positive $\alpha$. We now need to establish
\begin{equation}\label{the1st3}
\lim_{R\to\infty} \limsup_{\delta\to 0} \int_R^\infty L^{\Re\alpha-1} \vol\left\{ \vecx\in \scrD :  \delta^{n/(n+1)}  q_{\min}(\vecx,\delta,\scrA) \leq L \right\} dL = 0,
\end{equation}
\begin{equation}\label{the2nd3}
\lim_{r\to 0} \limsup_{\delta\to 0} \int_0^r L^{\Re\alpha-1} \vol\left\{ \vecx\in \scrD :  \delta^{n/(n+1)}  q_{\min}(\vecx,\delta,\scrA) \leq L  \right\} dL = 0.
\end{equation}
Here \eqref{the1st3} is immediate since the integrant is bounded above by $L^{\Re\alpha-1}$, and it remains to check \eqref{the2nd3}. Instead of \eqref{truu} we must now satisfy
\begin{equation} \label{truu2}
\hatZZ^{n+1}h(\vecx) a(\delta^{-n/(n+1)}) \cap L \fC(\sigma_1^{1/n} \scrA) \neq \emptyset ,
\end{equation}
which leads us to the volume $V(L,\delta^{-n/(n+1)})$ of $\vecx\in[0,1]^n$ so that we have an element in $\hatZZ^{n+1}h(\vecx) a(\delta^{-n/(n+1)})$ of norm at most $b L$, for some constant $b$ depending only on the choice of $\scrA$. 
If we set $y=\delta^{-n/(n+1)}>1$ and denote by $\chi_R$ the characteristic function of a ball of radius $R=b L<1$ centered at the origin, then the desired volume is bounded above by 
\begin{equation}\label{lalala}
\begin{split}
V(L,y) & \leq \int_{[0,1]^n} \sum_{\substack{(\vecp,q)\in\hatZZ^{n+1}\\ q\geq 0}} \chi_{bL}((\vecp,q) h(\vecx) a(y)) \,d\vecx \\
& = \int_{[0,1]^n} \sum_{\substack{(\vecp,q)\in\hatZZ^{n+1}\\ q>0}} \chi_{bL}((\vecp-q \vecx) y^{1/n} ,q y^{-1})  \,d\vecx \\
& = \int_{[0,1]^n} \sum_{\vecm\in\ZZ^n}  \sum_{\substack{(\vecp,q)\in\hatZZ^{n+1}\\ 0\leq p_j<q}} \chi_{bL}((\vecp+q\vecm-q \vecx) y^{1/n} ,q y^{-1})  \,d\vecx \\
& \leq  \int_{\RR^n} \sum_{q=1}^\infty q^n  \chi_{bL}(\vecx q y^{1/n} ,q y^{-1})  \,d\vecx \\
%& = y^{-1}  \int_{\RR^n} \sum_{q=1}^\infty \chi_L(\vecx ,q y^{-1})  \,d\vecx \\
%& = y^{-1}  \int_{\RR^n} \sum_{q=1}^\infty \chi_1(L^{-1}\vecx ,(L y)^{-1} q)  \,d\vecx \\
& = y^{-1} L^n  \int_{\RR^n} \sum_{q=1}^\infty \chi_b(\vecx ,(L y)^{-1} q) \, d\vecx \\
& \ll L^{n+1} 
\end{split}
\end{equation}
where the implied constant is independent of $0<L\leq b^{-1}$ and $y>1$. We conclude
\begin{equation}\label{CME}
\vol\left\{ \vecx\in \scrD :  \delta^{n/(n+1)}  q_{\min}(\vecx,\delta,\scrA) \leq L \right\}  \ll L^{n+1}
\end{equation}
for all $0<\delta< 1$ and $0<L\leq b^{-1}$. Hence \eqref{the2nd3} follows for $-(n+1)<\Re\alpha<0$.
\end{proof}

\section{Discrete sampling}\label{secDiscrete}

So far we have considered $\vecx$ as a random point uniformly distributed (with respect to the Lebesgue measure) in $\scrD\subset[0,1]^n$. We will replace this with a discrete sampling over points $\vecx_{\vecj,N}=N^{-1} \vecj$ in $\scrD$, with $\vecj$ ranging over $\ZZ^n$. We will also allow an additional shift by a fixed $\vecx_0\in\RR^n$.

\begin{prop}\label{PropLDdiscrete}
For $\scrA\subset\RR^n$ bounded and $\scrD\subset[0,1]^n$, both with boundary of Lebesgue measure zero and non-empty interior, $L>0$, $\vecx_0\in\RR^n$, $c>0$, we have
\begin{equation}\label{LD1aaaa}
\lim_{\substack{\delta\to 0, N\to\infty\\ c\delta^{-1}\leq N}} \frac{\#\left\{ \vecj \in \ZZ^n\cap N\scrD :  \delta^{n/(n+1)}  q_{\min}(\vecx_0+N^{-1} \vecj ,\delta,\scrA) > L  \right\}}{N^n \vol\scrD } = E_\scrA(L)
\end{equation}
with $E_\scrA(L)$ as in Proposition \ref{PropLD}.
\end{prop}

\begin{proof}
This follows from the same argument as the proof of Proposition \ref{PropLD}, if we replace the convergence of the void statistics for Farey fractions from \cite{Farey} with Proposition \ref{PropFarey1} in Section \ref{secPigeon}, a new result on pigeonhole statistics.
\end{proof}

In the one-dimensional case $n=1$, with $\delta=1/N$, $\scrD=[0,1)$, $\scrA=[0,1)$ and $\vecx_0=0$, \eqref{LD1aaaa} simplifies to 
\begin{equation}\label{LD1aaaabbb}
\lim_{N\to\infty} \frac{1}{N} \#\left\{ j=0,\ldots,N-1 :  N^{-1/2} q_{\min}(N^{-1} j ,N^{-1},[0,1)) > L \right\} = \int_L^\infty \eta(s) ds .
\end{equation}
Figure \ref{fig1} gives a comparison with numerical data for $N=3000$, which was generated by Mathematica with the input

\centerline{}
\centerline{\framebox{\includegraphics[width=0.95\textwidth]{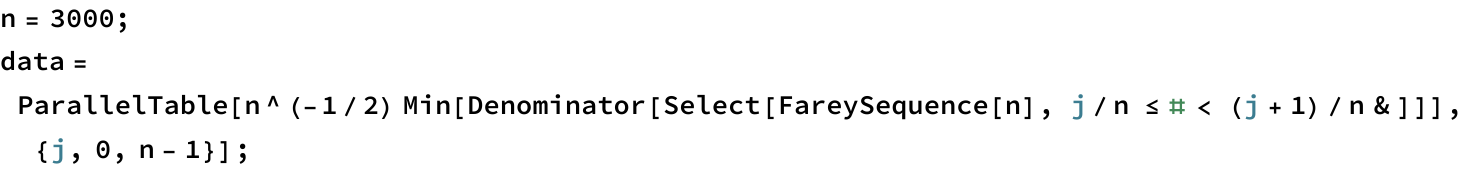}}}
\centerline{}

\noindent
We have here used the fact that $q_{\min}(N^{-1} j ,N^{-1},[0,1))\leq N$ for rationals in an interval of length $1/N$. 

Let us now turn to the convergence of moments.

\begin{prop}\label{propidiscrete}
For $\scrA\subset\RR^n$ bounded and $\scrD\subset[0,1]^n$, both with boundary of Lebesgue measure zero and non-empty interior, $\alpha\in\CC$ with $|\Re\alpha|<n+1$, $\vecx_0\in\RR^n$, $c>0$, we have
\begin{equation}\label{LD1Mom6789}
\lim_{\substack{\delta\to 0, N\to\infty\\ c\delta^{-1}\leq N}}  \frac{\delta^{\alpha n/(n+1)} }{N^n \vol\scrD} \sum_{\vecj \in \ZZ^n\cap N\scrD}   q_{\min}(\vecx_0+N^{-1} \vecj ,\delta,\scrA)^\alpha =
\int_0^\infty L^{\alpha} \; dE_\scrA(L).
\end{equation}
\end{prop}

\begin{proof}
The proof follows the same steps as for Proposition \ref{propi}, with the continuous average replaced by the discrete. 
The crucial step is to show that, for $0<\Re\alpha<n+1$,
\begin{equation}\label{the1st123456}
\lim_{R\to\infty} \limsup_{\substack{\delta\to 0, N\to\infty\\ c\delta^{-1}\leq N}} \int_R^\infty L^{\Re\alpha-1}  \frac{\#\left\{ \vecj \in \ZZ^n\cap N\scrD :  \delta^{n/(n+1)}  q_{\min}(\vecx_0+N^{-1} \vecj ,\delta,\scrA) > L  \right\}}{N^n \vol\scrD } dL = 0,
\end{equation}
and for $-(n+1)<\Re\alpha<0$,
\begin{equation}\label{the2nd3123456}
\lim_{r\to 0} \limsup_{\substack{\delta\to 0, N\to\infty\\ c\delta^{-1}\leq N}} \int_0^r L^{\Re\alpha-1}  \frac{\#\left\{ \vecj \in \ZZ^n\cap N\scrD :  \delta^{n/(n+1)}  q_{\min}(\vecx_0+N^{-1} \vecj ,\delta,\scrA) \leq L  \right\}}{N^n \vol\scrD } dL = 0.
\end{equation}

As to the former, let $\scrA_0$ be an open ball contained in $\scrA$. Since by assumption $c\delta^{-1}\leq N$, there is an $\epsilon\in(0,\frac{\delta}{2}]$ such that 
\begin{equation}
\vecr+\tfrac12\scrA_0 \subset \scrA_0 
\end{equation}
for every $\vecr\in \left[-\frac{\epsilon}{\delta N},\frac{\epsilon}{\delta N}\right]^n$, and therefore 
\begin{equation}
q_{\min}(\vecx_0+N^{-1} \vecj ,\delta,\scrA) \leq q_{\min}(\vecx_0+N^{-1} (\vecj+\vecr) ,\delta,\tfrac12 \scrA_0) .
\end{equation}
This implies
\begin{equation}\label{asin}
\begin{split}
& \frac{1}{N^n} \#\left\{ \vecj \in \ZZ^n\cap N\scrD :  \delta^{n/(n+1)}  q_{\min}(\vecx_0+N^{-1} \vecj ,\delta,\scrA) > L  \right\} \\
& \leq \left(\frac{\delta}{2\epsilon}\right)^n \int_{[-\frac{\epsilon}{\delta N},\frac{\epsilon}{\delta N}]^n} \#\left\{ \vecj \in \ZZ^n\cap N [0,1]^n :  \delta^{n/(n+1)}  q_{\min}(\vecx_0+N^{-1} (\vecj+\vecr) ,\delta,\tfrac12 \scrA_0) > L  \right\} \, d\vecr \\
%& \leq \left(\frac{\delta}{2\epsilon}\right)^n \int_{[-\frac{1}{2N},\frac{1}{2N}]^n} \#\left\{ \vecj \in \ZZ^n\cap N [0,1]^n :  \delta^{n/(n+1)}  q_{\min}(\vecx_0+N^{-1} (\vecj+\vecr) ,\delta,\tfrac12 \scrA_0) > L  \right\} \, d\vecr \\
& \leq 2 \left(\frac{\delta}{2\epsilon}\right)^n \vol\left\{ \vecx\in [0,1]^n :  \delta^{n/(n+1)}  q_{\min}(\vecx ,\delta,\tfrac12 \scrA_0) > L  \right\}  .
\end{split}
\end{equation}
We can now apply our previous estimate \eqref{bibo} to get the required upper bound. 

The case of negative $\alpha$ is analogous. We now take a ball $\scrA_1$ containing $\scrA$. There exists $\epsilon\in(0,\frac{\delta}{2}]$ such that 
\begin{equation}
\scrA_1 \subset \vecr+ 2\scrA_1 
\end{equation}
for every $\vecr\in \left[-\frac{\epsilon}{\delta N},\frac{\epsilon}{\delta N}\right]^n$,
and therefore 
\begin{equation}
q_{\min}(\vecx_0+N^{-1} (\vecj+\vecr) ,\delta,2 \scrA_0) \leq q_{\min}(\vecx_0+N^{-1} \vecj ,\delta,\scrA) .
\end{equation}
Following the same steps as in \eqref{asin}, we can now reduce to the continuous sampling estimate \eqref{CME}.
\end{proof}

Alternatively, we could have used for the proof of negative $\alpha$ the following counterpart of \eqref{lalala}, 
\begin{equation}\label{lalala2}
\frac{1}{N^n} \sum_{\vecj \in \ZZ^n\cap [0,N]^n} \sum_{\substack{(\vecp,q)\in\hatZZ^{n+1}\\ q>0}} \chi_{bL}((\vecp,q) (h(\vecx_0+N^{-1} \vecj ) a(y) )
 \ll L^{n+1} 
\end{equation}
uniformly for $1\leq y=\delta^{-n/(n+1)}\leq c^{-n/(n+1)} N ^{n/(n+1)}$. To prove this note that
\begin{equation}
\sum_{\substack{(\vecp,q)\in\hatZZ^{n+1}\\ q>0}} \chi_{bL}((\vecp-q \vecx) y^{1/n} ,q y^{-1}) )
\leq \sum_{\substack{(\vecp,q)\in\hatZZ^{n+1}\\ q>0}} \chi_{(1+n^{1/2}c^{-1})bL}((\vecp-q (\vecx+\vecr)) y^{1/n} ,q y^{-1}) ) ,
\end{equation}
provided $\|\vecr\|_\infty\leq N^{-1}$. Now take $\vecx=\vecx_0+N^{-1} \vecj$ and integrate $\vecr$ over the cube $[-\frac{N}{2},\frac{N}{2})^n$. This shows that the left hand side of \eqref{lalala2} is bounded above by the left hand side of \eqref{lalala} with $L$ replaced by $(1+n^{1/2}c^{-1})L$, and the claim \eqref{lalala2} is proved.

We remark that, for $\delta=1/N$, \eqref{LD1Mom6789} becomes
\begin{equation}\label{LD1Mom22}
\lim_{N\to\infty}  \frac{1}{N^{n+\alpha n/(n+1)} \vol\scrD} \sum_{\vecj \in \ZZ^n\cap N\scrD}   q_{\min}(\vecx_0+N^{-1} \vecj ,N^{-1},\scrA)^\alpha =
\int_0^\infty L^{\alpha} \; dE_\scrA(L)
\end{equation}
which, for $n=1$, $\alpha=1$, $\scrD=(0,1]$, $\scrA=(0,1]$ and $\vecx_0=0$, yields the Kruyswijk-Meijer conjecture \cite{KM77,Steward13,Balazard23}. %Recall \eqref{16} for the value $\frac{16}{\pi^2}$ of the limit.
A numerical comparison of the actual moments and the limit are plotted in Figure \ref{fig4}.

\begin{figure}
\begin{center}
\includegraphics[width=0.7\textwidth]{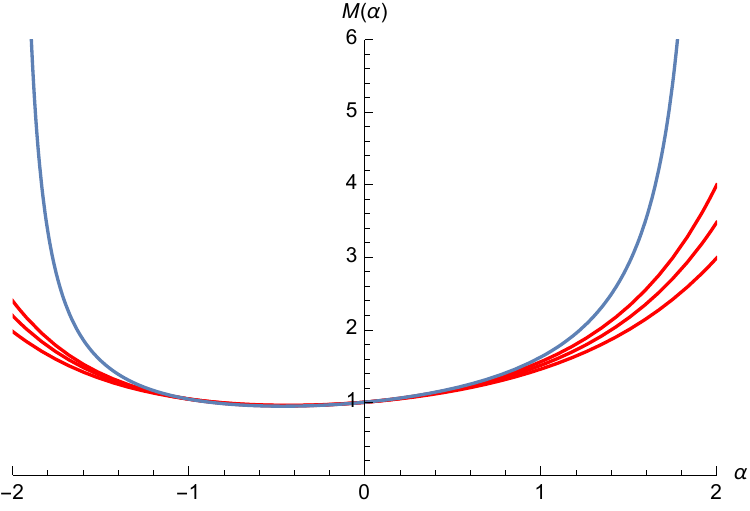}
\end{center}
\caption{The limiting moments $M(\alpha)$ for real $\alpha$ (blue) compared with finite-$N$ approximations (red) corresponding to the left hand side of \eqref{LD1Mom22} with $N=100$, $50$, $25$ (top to bottom) and $n=1$, $\alpha=1$, $\scrD=[0,1)$, $\scrA=[0,1)$, $\vecx_0=0$.} \label{fig4}
\end{figure}

\section{Pigeonhole statistics for Farey fractions}\label{secPigeon}

We start by recalling Theorem 3 in \cite{Farey} regarding the fine-scale statistics of Farey fractions. The case $k=0$ corresponds to the void statistics, which we used in the proof of Proposition \ref{PropLD} for the limit distribution of smallest denominators in the case of continuous sampling. 

\begin{prop}\label{PropFarey1}
For $\scrA\subset\RR^n$ bounded and $\scrD\subset[0,1]^n$, both with boundary of Lebesgue measure zero and non-empty interior, $k\in\ZZ_{\geq 0}$, we have
\begin{equation}\label{LD1F1}
\lim_{\delta\to 0} \frac{\vol\left\{ \vecx\in \scrD :  \#(\scrF_Q \cap \vecx+\sigma_Q^{-1/n} \scrA+\ZZ^n)=k \right\}}{\vol\scrD } = P(k,\scrA)
\end{equation}
where
\begin{equation}
P(k,\scrA) = \mu\{ g\in\GamG : \#(\hatZZ^{n+1}g \cap \fC(\scrA))  = k \} .
\end{equation}
\end{prop}

The following proposition will play the analogous role in the discrete sampling case, and provide the final ingredient of the proof of Proposition \ref{PropLDdiscrete}.

\begin{prop}\label{PropFarey2}
For $\scrA\subset\RR^n$ bounded and $\scrD\subset[0,1]^n$, both with boundary of Lebesgue measure zero and non-empty interior, $k\in\ZZ_{\geq 0}$, $\vecx_0\in\RR^n$, $c>0$, we have
\begin{equation}\label{LD1F2}
\lim_{\substack{Q, N\to\infty\\ c Q^{n+1}\leq N^n}} \frac{\#\left\{ \vecj \in \ZZ^n\cap N\scrD :  \#(\scrF_Q \cap \vecx_0+N^{-1} \vecj +\sigma_Q^{-1/n} \scrA+\ZZ^n)=k \right\}}{N^n \vol\scrD } = P(k,\scrA)
\end{equation}
with $P(k,\scrA)$ as in Proposition \ref{PropFarey1}.
\end{prop}

We refer to the above as ``pigeonhole statistics'' for the following reason. Take $\scrA=[0,s)^n$ for some given $s$, let $N$ run through the positive integers, and choose $Q=Q_N$ so that $\sigma_1 Q^{n+1}=s^n N^n$. Then the cubes (=pigeon holes) 
\begin{equation}
\vecx_0+N^{-1} \vecj +\sigma_Q^{-1/n} \scrA = \vecx_0+N^{-1} (\vecj + [0,1)^n)
\end{equation} 
tile $\TT^n:=\RR^n/\ZZ^n$, and the left hand side of \eqref{LD1F2} counts the number of cubes that contain exactly $k$ Farey points with denominator at most $Q_N$.

Following the strategy of proof of Proposition \ref{PropFarey1} (Theorem 3 in \cite{Farey}), we need to replace the equidistribution theorem for closed horospheres (Theorem 1 in \cite{Farey}) by the following discrete version. There has been significant interest recently in studying the distribution of rational points on horospheres. We refer the interested reader to \cite{Burrin22,Einsiedler21,Einsiedler16,ElBaz22,Pattison23} and references therein.  

\begin{prop}\label{equiThm2}
For $f:\TT^n\times\GamG\to\RR$ bounded continuous, $c>0$, we have
\begin{equation}\label{equiThm1eq}
	\lim_{\substack{N,Q\to\infty\\ cQ^{n+1}\leq N^n}} \frac{1}{N^n} \sum_{\vecj\in\ZZ^n/N\ZZ^n} f\big(N^{-1}\vecj ,h(\vecx_0+N^{-1}\vecj )a(Q)\big)  \\
	= \int_{\TT^n\times\GamG} f(\vecx,g) \, d\vecx \, d\mu(g) .
\end{equation}
\end{prop}

\begin{proof}
By a standard measure-theoretic argument, it will be sufficient to show that for $\scrD\subset\TT^n$ with boundary of Lebesgue measure zero and non-empty interior, $f:\GamG\to\RR$ bounded continuous, we have
\begin{equation}\label{equiThm1eq2}
	\lim_{\substack{N,Q\to\infty\\ cQ^{n+1}\leq N^n}} \frac{1}{N^n\vol\scrD} \sum_{\vecj\in\ZZ^n/N\ZZ^n\cap N\scrD} f\big(h(\vecx_0+N^{-1}\vecj )a(Q)\big)  \\
	= \int_{\GamG} f(g) \, d\mu(g) .
\end{equation}
For a given sequence of $(N_i,Q_i)$, the left hand side of \eqref{equiThm1eq2} defines a sequence of probability measures $\nu_i$ on $\GamG$ via 
\begin{equation}
\nu_i(f) = \frac{1}{\#(\ZZ^n/N_i\ZZ^n\cap N_i\scrD)} \sum_{\vecj\in\ZZ^n/N_i\ZZ^n\cap N_i\scrD} f\big(h(\vecx_0+N_i^{-1}\vecj )a(Q_i)\big)
\end{equation}
which we need to show converges weakly to the probability measure $\mu$. By Mahler's compactness criterion for the space of lattices, the complement of large-volume compact sets are characterised by lattices with short vectors. The estimate \eqref{lalala2} therefore shows that $(\nu_i)_i$ is tight and thus each subsequence contains a convergent subsequence. We may now assume (without loss of generality) that $f$ has compact support and is therefore uniformly continuous. A key observation is that
\begin{equation}
h(\vecx_0+N^{-1}\vecj )a(Q) = h(\vecx_0)a(Q) h(Q^{1+1/n} N^{-1}\vecj ) .
\end{equation}
In the following we restrict to subsequences along which $Q_i^{1+1/n} N_i^{-1} \to \tau_0$ for some $\tau_0\in [0,c^{-1/n}]$. 

In the case $\tau_0=0$ the discrete average is uniformly close to the continuous average (by uniform continuity of $f$), and thus by Theorem 1 in \cite{Farey} the limit is given by $\mu$. 

    If $\tau_0>0$, then any weak  limit is invariant under the map $\GamG\to\GamG$, $\Gamma g \mapsto \Gamma g h(\tau_0\vecj)$ for any $\vecj\in\ZZ^n$. Since the action of $G$ on $\GamG$ by right multiplication is mixing with respect to $\mu$, we have in particular that the action of the subgroup $H_{\tau_0}=\{ h(\tau_0\vecj) : \vecj\in\ZZ^n\}$ is $\mu$-ergodic. There are various avenues to determine the possible limit points of $(\nu_i)_i$, for example referring to disjointness results for mixing actions. Here we take a more direct path shown to me by M.~Einsiedler. Let us $\epsilon$-broaden the probability measure $\nu_i$ by setting, for $0<\epsilon<\tau_0$,
\begin{equation}
\nu_i^\epsilon(f) = \nu_i(f_\epsilon), \qquad f_\epsilon(g):=\frac{1}{\epsilon^n} \int_{[-\frac{\epsilon}{2},\frac{\epsilon}{2}]^n} f\big(g h(\vecx)\big)\, d\vecx .
\end{equation}
We also define the probability measure corresponding to the continuous horospherical average
\begin{equation}
\mu_i(f) = \frac{1}{\vol \scrD} \int_{\scrD} f\big(h(\vecx)a(Q_i)\big)\, d\vecx ,
\end{equation}
and the complementary probability measure 
\begin{equation}
\overline\nu_i^\epsilon = \frac{\mu_i - \epsilon^n \nu_i^\epsilon}{1  - \epsilon^n} .
\end{equation}
Suppose $\nu_i^\epsilon\to \nu^\epsilon$ along a converging subsequence. As we have $\mu_i \to \mu$ (along any subsequence), by construction $\overline\nu_i^\epsilon\to \overline\nu^\epsilon$ along the same subsequence as $\nu_i^\epsilon$, and the limits satisfy the relation
\begin{equation}
\epsilon \nu^\epsilon + (1-\epsilon) \overline\nu^\epsilon = \mu.
\end{equation}
All three limit measures are $H_{\tau_0}$-invariant. Since the action of $H_{\tau_0}$ is $\mu$-ergodic, by the extremality property of ergodic measures, we conclude $\nu^\epsilon = \overline\nu^\epsilon = \mu$ for every given $\epsilon>0$. Because $f$ is uniformly continuous, we have
\begin{equation}
\lim_{\epsilon\to 0} \sup_i |\nu_i(f) - \nu_i^\epsilon (f)| =0 
\end{equation}
and thus every limit point of $(\nu_i)_i$ must be equal to $\mu$.
\end{proof}

\section{Moments of the distance function for the Farey sequence}\label{sec97}

It is instructive to compare the moments of the smallest denominator of \cite{Chen23} with the distance of a random point to the Farey sequence in \cite{Kargaev97}. Let us already consider the higher dimensional distance function, for $\vecx\in\TT^n$, $Q\geq 1$,
\begin{equation}
\dist(\vecx,\scrF_Q) = \min\{ \|\vecx+\vecr+\vecm\| : \vecr\in\scrF_Q,\;\vecm\in\ZZ^n \} ,
\end{equation}
where $\|\,\cdot\,\|$ can be any of the standard norms on $\RR^n$. We first of all note that, for every $s\geq 0$,
\begin{equation} \label{cal10F}
\sigma_Q^{1/n} \dist(\vecx,\scrF_Q) > s  \Leftrightarrow \scrF_Q \cap \vecx+\sigma_Q^{-1/n} \scrB_s +\ZZ^n =\emptyset ,
\end{equation}
where $\scrB_s=\{\vecy\in\RR^n : \| \vecy \| \leq s \}$, which gives the connection with the void statistics with $\scrA=\scrB_s$, and via \eqref{cal10} to the smallest denominator. Note that for finite $Q$ resp.\ $\delta$ the moments of the two distributions are different, although the limiting distributions are the same, up to a simple scaling. The following statement generalises the results of \cite[Theorem 1.4]{Kargaev97} in dimension $n=1$ (in the second of the four ranges, which corresponds to $-1<\beta<1$ in the statement below) to arbitrary dimension. 

 \begin{prop}\label{propiF}
For $\scrD\subset[0,1]^n$ with boundary of Lebesgue measure zero and non-empty interior, $\beta\in\CC$ with $|\Re\beta| <n$, we have
\begin{equation}\label{LD1Mom2345F}
\lim_{Q\to\infty } \frac{\sigma_Q^{\beta/n}}{\vol\scrD} \int_{\scrD}   \dist(\vecx,\scrF_Q)^\beta d\vecx=
\int_0^\infty s^{\beta} \; dF_{\scrB_1}(s),
\end{equation}
with $F_{\scrB_1}(s)=P(0,\scrB_s)=E_{\scrB_1}(\sigma_1^{-1/(n+1)} s^{n/(n+1)})$.
\end{prop}

\begin{proof}
The strategy is the same as for the proof of Proposition \ref{propi}. The key point is to show that when $\Re\beta>0$,
\begin{equation}\label{the1stF}
\lim_{R\to\infty} \limsup_{Q\to\infty} \int_R^\infty s^{\Re\beta-1} \vol\left\{ \vecx\in \scrD :  \sigma_Q^{1/n}  \dist(\vecx,\scrF_Q) > s \right\} ds = 0.
\end{equation}
Since $Q^{-1}\ZZ^n \cap[0,1)^n \subset \scrF_Q$ we have $\dist(\vecx,\scrF_Q) \leq C Q^{-1}$ for some constant $C$ (depending only on the choice of norm $\|\,\cdot\,\|$) and hence can restrict the integral to $s\leq C \sigma_Q^{1/n} Q^{-1} = C \sigma_1^{1/n} Q^{1/n}$. 
We need to estimate the volume of $\vecx$ such that
\begin{equation}
\hatZZ^{n+1}h(\vecx) a(Q) \cap \fC(\scrB_s) =\emptyset ,
\end{equation}
which is equivalent to
\begin{equation}
\hatZZ^{n+1}h(\vecx) a(Q) a(s^{-n/(n+1)}) \cap s^{n/(n+1)} \fC(\scrB_1) =\emptyset .
\end{equation}
We now apply \eqref{bibo22} with $y=Q s^{-n/(n+1)}$ in place of $Q$, and $L=s^{n/(n+1)}$. (Note here that, since $s\leq C \sigma_1^{1/n} Q^{1/n}$ we have $y=Q s^{-n/(n+1)}\geq C^{-n/(n+1)} \sigma_1^{-1/(n+1)} Q^{1-1/(n+1)}>1$ for all sufficiently large $Q$.) This yields
\begin{equation}\label{lord}
\vol\left\{ \vecx\in \scrD :  \sigma_Q^{1/n}  \dist(\vecx,\scrF_Q) > s \right\} \ll s^{-n} .
\end{equation}
Now \eqref{lord} implies \eqref{the1stF}.

In the case $\Re\beta<0$ we need to check that
\begin{equation}\label{the2nd3F}
\lim_{r\to 0} \limsup_{Q\to\infty} \int_0^r s^{\Re\beta-1} \vol\left\{ \vecx\in \scrD :  \sigma_Q^{1/n}  \dist(\vecx,\scrF_Q) \leq s  \right\} ds = 0.
\end{equation}
which leads to the condition
\begin{equation} \label{truu2F}
\hatZZ^{n+1}h(\vecx) a(Q) a(s^{-n/(n+1)}) \cap s^{n/(n+1)} \fC(\scrB_1) \neq \emptyset ,
\end{equation}
hence the above lattice has an element of norm at most $s^{n/(n+1)}$. From \eqref{lalala} applied to $L=s^{n/(n+1)}$ and $y=Q s^{-n/(n+1)}>1$ we conclude
\begin{equation}
\vol\left\{ \vecx\in \scrD :  \sigma_Q^{1/n}  \dist(\vecx,\scrF_Q) \leq s \right\}  \ll s^n 
\end{equation}
for all $0<s\leq 1$. This proves \eqref{the2nd3F}.
\end{proof}

For completeness, we also state the counterpart of Proposition \ref{propiF} in the case of discrete sampling, which may be viewed as the moments of the pigeonhole void distribution.

\begin{prop}\label{propidiscreteF}
For $\scrD\subset[0,1]^n$ with boundary of Lebesgue measure zero and non-empty interior, $\beta\in\CC$ with $|\Re\beta| <n$, $\vecx_0\in\RR^n$, $c>0$, we have
\begin{equation}\label{LD1Mom6789F}
\lim_{\substack{N,Q\to\infty\\ cQ^{n+1}\leq N^n}} \frac{\sigma_Q^{\beta/n}}{N^n} \sum_{\vecj\in\ZZ^n/N\ZZ^n}   \dist(\vecx_0+N^{-1} \vecj ,\scrF_Q)^\beta =
\int_0^\infty s^{\beta} \; dF_{\scrB_1}(s) .
\end{equation}
\end{prop}

The proof of this statement follows along the same lines as Proposition \ref{propidiscrete}.

\section{Minimal resonance orders}\label{minsec}

The paper \cite{Meiss21} which motivated \cite{Chen23} was in fact interested in the distribution of the minimal resonance order
\begin{equation}
M(\vecomega,\rho) = \min_{\vecp\in\ZZ^n\setminus\{0\}} \left\{ \| \vecp \|_1 :  \min_{q\in\ZZ} \Delta_{\vecp,q}(\vecomega) \leq \rho \right\}
\end{equation}
in the limit of small $\rho$,
where
\begin{equation}
\Delta_{\vecp,q}(\omega) =\frac{|\vecp\cdot\vecomega - q|}{\| \vecp\|_2} .
\end{equation}
This quantity appears naturally in the study of the breakdown of invariant tori in integrable dynamical systems under perturbation.
In the same vein as our previous discussion, the key fact is that
\begin{equation} \label{cal100011}
M(\vecomega,\rho) > L \rho^{-1/(n+1)} \Leftrightarrow
 \left\{ (\vecp,q)\in\ZZ^{n+1} :  0<\| \vecp \|_1 \leq L \rho^{-1/(n+1)},\; \Delta_{\vecp,q}(\omega) \leq \rho \right\} =\emptyset .
\end{equation}
The last statement is equivalent to
\begin{equation}
 \left\{ (\vecp,q)\in\ZZ^{n+1}\setminus\{ \vecnull \} :  \| \vecp \|_1 \leq L \rho^{-1/(n+1)},\; |\vecp\cdot\vecomega - q|\leq \rho \| \vecp\|_2  \right\} =\emptyset ,
\end{equation}
which is equivalent to
\begin{equation}
 \left\{ (\vecp,q)\in\hatZZ^{n+1} :  \| \vecp \|_1 \leq L \rho^{-1/(n+1)},\; |\vecp\cdot\vecomega - q|\leq \rho \| \vecp\|_2  \right\} =\emptyset .
\end{equation}
This in turn can be written as
\begin{equation}\label{lala}
\hatZZ^{n+1} \cap \fB_L \begin{pmatrix} \rho^{-1/(n+1)} 1_n & \trans\vecnull \\ \vecnull & \rho^{n/(n+1)} \end{pmatrix} 
\begin{pmatrix} 1_n & \trans\vecomega \\ \vecnull & 1\end{pmatrix} = \emptyset,
\end{equation}
where
\begin{equation}
\fB_L=\{ (\vecx,y)\in \RR^{n+1} : \|\vecx\|_1\leq L, \; |y|\leq \|\vecx\|_2  \} .
\end{equation}
The problem at hand is thus the distribution of visible lattice points in thin, randomly sheared sets, and a close variant of  the fine-scale statistics of linear forms studied in \cite{npoint}. In the present setting we can directly apply \cite[Theorem 6.7, $\vecalf=\vecnull$]{partI} to \eqref{lala}, which yields the following. 

\begin{prop}
Let $L>0$ and $\lambda$ a Borel probability measure on $\RR^n$ that is absolutely continuous with respect to Lebesgue measure.
Then
\begin{equation}\label{eqr}
\lim_{\rho\to 0} \lambda\left\{ \vecomega : \rho^{1/(n+1)}  M(\vecomega,\rho) > L \right\} 
= R(L) 
\end{equation}
with
\begin{equation}
R(L) = \mu\{ g\in\GamG : \hatZZ^{n+1}   g \cap \fB_L  = \emptyset\} . 
\end{equation}
\end{prop}

In previous sections we chose a random vector uniformly distributed in $\scrD$, bounded with boundary of measure zero, while here the law of $\vecx$ is given by an absolutely continuous probability measure $\lambda$. This was done to have a direct reference to the theorems quoted. A standard measure-theoretic argument shows that the two formulations are equivalent.

%***what about the tail? $\fB_L$ is not one of the sets studies in \cite{Strombergsson11}***

The density of $R(L)$ corresponds to the histogram in Fig. 8(a) of \cite{Meiss21} for $n=2$. The exponent of $\rho$  in Fig. 8(b) is $\approx 0.334$, which is consistent with the theoretically predicted scaling with exponent $1/(n+1)$ in \eqref{eqr}.

For further variations on the techniques reviewed here, including applications to integrable dynamics, we refer the reader to \cite{Dettmann17}.

\end{document}